\documentclass[12pt]{amsart}

\usepackage{fullpage}
\usepackage{amsmath}
\usepackage{amssymb}
\usepackage{amsthm}
\usepackage{amstext}
\usepackage{appendix}
\usepackage{perpage}
\usepackage[x11names]{xcolor}
\usepackage{tikz-cd}
\usepackage{ytableau}\ytableausetup{centertableaux}
\usepackage[colorlinks=true, allcolors=DodgerBlue4]{hyperref}

\newtheorem{thm}{Theorem}[section]

\newtheorem{coro}[thm]{Corollary}
\newtheorem{prop}[thm]{Proposition}

\theoremstyle{definition}
\newtheorem{expl}[thm]{Example}
\newtheorem{defi}[thm]{Definition}
\newtheorem{remark}[thm]{Remark}
\newtheorem{quest}[thm]{Question}
\newtheorem{notation}[thm]{Notation}

\begin{document}

\title{Intersections of Deligne--Lusztig varieties and Springer fibres}

\author{Zhe Chen}

\address{Department of Mathematics, Shantou University, Shantou, China}

\email{zhechencz@gmail.com}

\begin{abstract}
In this paper we prove a direct geometric relation between Deligne--Lusztig varieties and Springer fibres in type $\mathsf{A}$: For any rational unipotent element, the Springer fibre cuts out a unique component of a specific Deligne--Lusztig variety; moreover, this component forms an open dense subset of a component of the Springer fibre. This boils down to a map from the unipotent variety to the Weyl group, and combines several constructions with a combinatorial flavour (like Weyr normal forms, Robinson--Schensted correspondence, and Spaltenstein's and Steinberg's labellings); it also provides a geometric interpretation of a classical dimension formula of unipotent centralisers. 
\end{abstract}

\maketitle

\tableofcontents

\section{Introduction}\label{section:Intro}

Let $\mathbb{G}$ be a connected reductive group over a finite field $\mathbb{F}_q$, and let $F$ be the geometric Frobenius endomorphism on $G:=\mathbb{G}\times_{\mathbb{F}_q}\overline{\mathbb{F}}_q$. Fix an $F$-stable Borel subgroup $B\subseteq G$ and an $F$-stable maximal torus $T\subseteq B$. There are two important classes of varieties lying in the flag variety $G/B$, namely, Deligne--Lusztig varieties and Springer fibres.

\vspace{2mm} Deligne--Lusztig varieties $X_w$ (see Definition~\ref{defi:DL}) are parametrised by $w\in W(T):=N(T)/T$, and their $\ell$-adic cohomology (with coefficients in suitable local systems) affords all the irreducible representations of $G^F=\mathbb{G}(\mathbb{F}_q)$. Meanwhile, Springer fibres $\mathcal{B}_u$ (see Definition~\ref{defi:springer}) are parametrised by the unipotent elements $u\in G$, and their $\ell$-adic cohomology affords all the irreducible representations of $W(T)$. Since their births in the seminal works \cite{DL1976} and \cite{Springer_1976_TrigSum}, respectively, these two families of varieties play crucial roles in the study of representations of Lie type groups and Weyl groups. In this paper we give a study of their relations in the case of type $\mathsf{A}$. 

\vspace{2mm} In the remaining part of this paper we take $\mathbb{G}$ to be $\mathrm{GL}_n$, and let $B$ be the standard upper Borel subgroup and $T$ the diagonal maximal torus. 

\vspace{2mm} Indeed recently we found that, for a specific unipotent $u$, the intersection $\mathcal{B}_{u,w}:=X_w\cap \mathcal{B}_u$ appears in the study of smooth representations of the profinite group $\mathrm{GL}_{d}(\mathbb{F}_q[[\pi]])$, where $d$ is a divisor of $n$ (see \cite{Chen_2020_twistetcent}); this serves the initial motivation for our attention on the relations between these two varieties. On the other hand, on the level of representations there are already striking relations found between these two constructions (see for example \cite{Lusztig_Green_Functions_CharSh} and \cite[1.3.5]{Bezru--Varshav_2021_L-packets}). In this paper, instead of representations we focus on the geometric relations. We remark that a similar theme has been considered over $\mathbb{C}$ in \cite{Tymoczko_LinearConditionFlag}, in which case Hessenberg varieties and Schubert cells took the roles of Springer fibres and Deligne--Lusztig varieties. In any case, providing the facts that both $X_w$ and $\mathcal{B}_u$ are vital in geometric representation theory, and that they share the same ambient space $G/B$, it is very interesting and natural to seek their geometric interactions. 

\vspace{2mm}  We found the following surprising simple relation between the components: (See Theorem~\ref{thm:main1} for the formal statement.)
\begin{itemize}
\item[(a)] Any Springer fibre at a rational unipotent element of $G$ has a component containing a component of a specific Deligne--Lusztig variety as a dense open subset; moreover, this Deligne--Lusztig component forms the whole intersection. 
\item[(b)] Conversely, every component of a Deligne--Lusztig variety at an involution of a specific shape is a dense open subset of a component of some Springer fibre.
\end{itemize}
As we will see in Theorem~\ref{thm:main1}, this component relation boils down to a map from the unipotent variety to the Weyl group.

\vspace{2mm} In Section~\ref{sec:DL et Spr} we make some preparations on Deligne--Lusztig varieties and Springer fibres.

\vspace{2mm} In Section~\ref{sec:prelim} we give a brief recall of some ingredients used in the proof of Theorem~\ref{thm:main1}, like Weyr normal forms, Robinson--Schensted correspondence, and Spaltenstein's and Steinberg's descriptions of Springer components.

\vspace{2mm} In Section~\ref{sec:components} we present the proof of Theorem~\ref{thm:main1}, which is a careful combination of the above constructions. We also derive a geometric proof of a classical dimension formula of unipotent centralisers (see Corollary~\ref{coro:dim of unip centraliser}).

\vspace{2mm} In Section~\ref{sec:further remarks} we give a few remarks, including three illustrating examples (one concerning the boundary of the theorem, one concerning a uniqueness property, and one concerning an opposite phenomenon), and a short discussion on the representations associated with $\mathcal{B}_{u,w}$.

\vspace{2mm} Throughout this paper: We use the convention notation $g^h={^{h^{-1}}g}=h^{-1}gh$ for elements $g,h$ in an algebraic group; all varieties are assumed to be reduced; by a component we always mean an irreducible component.

\vspace{2mm} \noindent {\bf Acknowledgement.} The author thanks George Lusztig and Alexander Stasinski for helpful comments and suggestions, and thanks Guangyi Yue for a helpful communication. During the preparation of this work the author is partially supported by the NSFC funding no.12001351.

\section{Deligne--Lusztig varieties and Springer fibres}\label{sec:DL et Spr}

In this section we recall some basics of Deligne--Lusztig varieties and Springer fibres. The details can be found in \cite{DL1976}, \cite{Shoji1988GeomOrbits}, \cite{Steinberg_occurrence},  \cite{Carter1993FiGrLieTy}.

\begin{defi}\label{defi:DL}
Let $w\in W(T)\cong S_n$. Then the Deligne--Lusztig variety at $w$ is 
$$X_w:=L^{-1}(BwB)/B,$$
where $L\colon G\rightarrow G$ is the Lang isogeny given by $g\mapsto g^{-1}F(g)$.
\end{defi} 

Viewing $G/B$ as the variety of complete flags, one can describe $X_w$ in the following way.

\begin{defi}\label{defi:relative posi}
Let 
$$\mathcal{F}\colon V_0\subseteq V_1\subseteq...\subseteq V_{n} \quad \textrm{and} \quad \mathcal{F}'\colon V'_0\subseteq V'_1\subseteq...\subseteq V'_{n}$$ 
be two complete flags of $V:=(\overline{\mathbb{F}}_q)^n$, with $\dim V_i=\dim V'_i=i$. We say that $\mathcal{F}$ and $\mathcal{F}'$ are in relative position $w\in W(T)=S_{n}$, if 
$$\dim V_i\cap V'_j=\# \left(    \{1,...,i\}\cap \{w(1),...,w(j)\}  \right)$$
for any $i,j\in\{1,...,n\}$. 
\end{defi}

\begin{itemize}\hypertarget{condition (*)}{}
\item[{\bf (*)}] Let $G=\mathrm{GL}_n$ be viewed as the automorphism group of $V$. Then $X_w$ is the variety consisting of the complete flags $\mathcal{F}\colon V_0\subseteq V_1\subseteq...\subseteq V_{n}$ such that $\mathcal{F}$ and $F\mathcal{F}$ are in relative position $w$, namely
$$\dim V_i\cap FV_j=\# \left( \{1,...,i\}\cap \{w(1),...,w(j)\} \right)$$
for any $i,j\in\{1,...,n\}$.
\end{itemize}

\begin{prop}[Deligne--Lusztig]\label{prop:DL property}
The variety $X_w$ is a smooth locally closed subvariety of $G/B$ of pure dimension $l(w)$, where $l(w)$ denotes the length of $w$.
\end{prop}
\begin{proof}
See \cite[Page~107]{DL1976}.
\end{proof}

\begin{defi}\label{defi:springer}
Let $u\in G$ be a unipotent element. Then the Springer fibre at $u$ is
$$\mathcal{B}_u:=\{ gB\in G/B \mid u^g \in B \}.$$ 
Using the term of flags, $\mathcal{B}_u\subseteq G/B$ can be viewed as the closed subvariety consisting of complete flags $\mathcal{F}\colon V_0\subseteq V_1\subseteq...\subseteq V_{n}$ stabilised by $u$ (namely, $uV_i=V_i$ for all $i$).
\end{defi}

Unlike $X_w$, usually $\mathcal{B}_u$ is singular, but one still has:

\begin{prop}[Spaltenstein, Steinberg]\label{prop:spr property}
The variety $\mathcal{B}_u$ is of pure dimension $v_G-\frac{1}{2}\dim C(u)$, where $v_G$ denotes the number of positive roots and $C(u)$ denotes the conjugacy class of $u$.
\end{prop}
\begin{proof}
See e.g.\ \cite[1.2]{Shoji1988GeomOrbits}.
\end{proof}

In this paper we will very often take the viewpoint that elements in $X_w$ and $\mathcal{B}_u$ are flags. We put $\mathcal{B}_{u,w}:=\mathcal{B}_u\cap X_w=\mathcal{B}_u\times_{G/B} X_w$.

\section{Some preliminaries}\label{sec:prelim}

In this section we recall some ingredients needed in the proof of Theorem~\ref{thm:main1}. We first fix the notation that will be used throughout this paper:

\begin{notation}\label{notation:unip notations}
Let $u\in G$ be a unipotent element. Then
\begin{itemize}
\item $J(u)=\mathrm{diag}\{ J_1,..., J_d \}$ is the standard Jordan normal form of $u$, where the $J_i$'s are the Jordan blocks (with non-increasing sizes). 

\item $r_i$ is the size of the Jordan block $J_i$ (that is, $J_i$ is an $r_i\times r_i$-matrix); in particular $r_i\geq r_{i+1}$.

\item $\lambda(u)$ is the Young diagram associated with $J(u)$, that is, a Young diagram whose $i$-th row has $r_i$ boxes.

\item $c_i$ is the number of boxes in the $i$-th column of $\lambda(u)$. For convenience, we also put $c_0=0$. Note that there are totally $r_1$ columns and $c_1=d$ rows.
\end{itemize}
\end{notation}

One ingredient we would need is the so-called Weyr normal form $W(u)$, which is a ``dual'' of the Jordan normal form $J(u)$:
\begin{defi}\label{defi:Weyr}
For a unipotent $u\in\mathrm{GL}_n(\overline{\mathbb{F}}_q)$, the matrix $W(u)$ is blocked upper triangular, and is characterised by the following rules:
\begin{itemize}
\item[(i)] The $i$-th diagonal block is the $c_i\times c_i$ identity matrix.

\item[(ii)] The block just right to the $i$-th diagonal block is of the form $\left(\begin{smallmatrix} I\\ 0 \end{smallmatrix}\right)$, where $I$ denotes the $c_{i+1}\times c_{i+1}$-identity matrix, and $0$ denotes the zero matrix of a suitable size.

\item[(iii)] All other blocks are zero.
\end{itemize}
So $W(u)$ is a blocked matrix of the shape
\begin{equation*}
\begin{bmatrix}
I & \left(\begin{smallmatrix} I\\ 0 \end{smallmatrix}\right) &  0 & 0 & ... & 0\\
0 & I & \left(\begin{smallmatrix} I\\ 0 \end{smallmatrix}\right) & 0 & ... & 0\\
0 & 0 & I & \left(\begin{smallmatrix} I\\ 0 \end{smallmatrix}\right) & ... & 0\\
... & ... & ...  & ... & ... & ... \\
0 & ... & 0  & 0 & I & \left(\begin{smallmatrix} I\\ 0 \end{smallmatrix}\right) \\
0 & 0 & ... & 0 & 0 & I\\
\end{bmatrix},
\end{equation*}
where the $I$'s denote some identity matrices of possibly different sizes and the $0$'s denote some zero matrices of possibly different sizes. 
\end{defi}

Although Weyr normal forms and Jordan normal forms were both discovered  in the second half of the 19th century, the Weyr form appears to be much lesser known; a comprehensive reference on these normal forms is \cite{Omeara_et_al_WeyrBk}. Note that recently there has been a (very different) application of Weyr normal forms in the representation theory of Lie type groups over local rings; see \cite{Stasinski_rep_SL_n_length_two}. 

\begin{remark}
One of the main features we need from Weyr form (instead of Jordan form) is that the block sizes are with respect to the columns of Young diagrams, which allows one to combine the other elements in the proof of Theorem~\ref{thm:main1} in a natural way.
\end{remark}

\begin{prop}\label{prop:weyr}
The unipotent elements $u$ and $W(u)$ are in the same conjugacy class.
\end{prop}
\begin{proof}
See e.g.\ \cite[2.2.2]{Omeara_et_al_WeyrBk}.
\end{proof}

\vspace{2mm} Another ingredient we need is the Robinson--Schensted correspondence. Recall that a very basic property of finite group representation theory is the identity $\sum_{\rho} (\dim \rho)^2=\#H$, where $H$ is a finite group and $\rho$ suns over the irreducible representations. If we take $H=S_n$, then this identity can be re-written as
$$\sum_{\lambda}(\#T(\lambda))^2=\#S_n,$$
where $\lambda$ runs over the Young diagrams of $n$ boxes, and $T(\lambda)$ denotes the set of standard $\lambda$-tableaux (we use the convention that the numbers in a standard tableaux go increasingly from left to right and from up to down). The Robinson--Schensted correspondence, which was later generalised by Knuth to a more general situation, gives a combinatorial explanation of this identity.

\begin{prop}[Robinson--Schensted correspondence]\label{prop:RS}
There is an explicit computable bijection between the sets
$$w(-,-)\colon \bigsqcup_{\lambda} T(\lambda)\times T(\lambda) \longrightarrow S_n,$$
satisfying the property $w(P,Q)=w(Q,P)^{-1}$, given by the following algorithm:
\begin{itemize}
\item[(i)] Take $(P,Q)\in  T(\lambda)\times T(\lambda)$. If $n$ is in the $(i,j)$-th box of $Q$, then we remove this box from $Q$, and denote the new tableau by $Q'$.
\item[(ii)] Suppose the number in the $(i,j)$-th box of $P$ is $n'$. We remove this box from $P$ and move $n'$ up by one row to replace the largest number smaller than $n'$.
\item[(iii)] Suppose the number replaced by $n'$ is $n''$, then we move $n''$ up by one row to replace the largest number smaller than $n''$, and so on, until we replaced a number in the first row.
\item[(iv)] Denote the number been replaced from the first row by $w(n)$, and denote the resulting tableau by $P'$.
\item[(v)] Repeat the above process for $n-1$ with the tableaux pair $(P',Q')$, and so on, until we find all $w(n),w(n-1),...,w(1)$. Then
\begin{equation*}
w(P,Q):=
\begin{pmatrix}
1    & 2    & \cdots &  n \\
w(1) & w(2) & \cdots & w(n)
\end{pmatrix}.
\end{equation*}
One often use the word notation $w(P,Q)=w(1)...w(n)$.
\end{itemize}
\end{prop}

\begin{proof}
See e.g.\ \cite[5.1.4]{Knuth_vol3}.
\end{proof}

Besides the above two ingredients, we also need the Young tableaux labelling of components of the Springer fibre $\mathcal{B}_u$ given in \cite{Spaltenstein_fixpt_uniptransform_flagmfld} and \cite{Steinberg_desingularisation_1976}. We follow Steinberg's description in \cite{Steinberg_occurrence}.

\begin{prop}[Tableaux labelling of components]\label{prop:Spal--Stei}
There is a bijection
$$T(\lambda(u)) \longleftrightarrow \left\{ \textrm{components of}\ \mathcal{B}_u \right\}.$$
More explicitly, for a given tableau $P\in T(\lambda(u))$, the corresponding component is characterised as the closure of an open subset $C(P)$, where $C(P)$ consists of the flags 
$$\mathcal{F}\colon V_0\subset V_1\subset...\subset V_n=V$$ 
constructed via the following steps:
\begin{itemize}
\item[(I)] Let $N:=u-I\in\mathfrak{gl}_n$, the nilpotent element associated with $u$;
\item[(II)] if $n$ is in the position $(i,j)=(i,r_i)=(c_j,j)$ of $P$, then $V_{n-1}$ is any hyperplane satisfying that
\begin{equation*}
\begin{cases}
NV_n+\mathrm{Ker}N^{j-1}\subseteq V_{n-1}  \\
NV_n+\mathrm{Ker}N^j\nsubseteq V_{n-1}  \\
\end{cases};
\end{equation*}
\item[(III)] once such a $V_{n-1}$ is chosen, we repeat the above process to construct a $V_{n-2}$ (by replacing $V_n,N...$ by $V_{n-1},N|_{V_{n-1}}$), and so on.
\end{itemize}
\end{prop}
\begin{proof}
See \cite[Section~2]{Steinberg_occurrence}.
\end{proof}

Moreover, they proved the following property of generic relative position for the components of $\mathcal{B}_u$:

\begin{prop}[Generic relative position]\label{prop:stei}
Let $P$ and $Q$ be two standard $\lambda(u)$-tableaux. Then there is an open dense subsvariety $X\subseteq \overline{C(P)}\times \overline{C(Q)}$ such that any closed point $(\mathcal{F}_1,\mathcal{F}_2)\in X$ has the relative position $w(P,Q)$.
\end{prop}
\begin{proof}
See \cite[Section~3]{Steinberg_occurrence} or \cite[II.9]{Spaltenstein_1982_Borel_bk}.
\end{proof}

\section{Distribution of components}\label{sec:components}

Our main theorem is:

\begin{thm}\label{thm:main1}
Let  $\mathcal{U}$ be the variety of unipotent elements of $G$. Then we have:
\begin{itemize}
\item[(a)] There is a canonical map 
$$\beta \colon \mathcal{U}^F\longrightarrow W(T)$$ 
such that, for $u\in\mathcal{U}^F$, the Springer fibre $\mathcal{B}_u$ intersects exactly one component of the Deligne--Lusztig variety $X_{\beta(u)}$, and this component is an open dense subset of an irreducible component of $\mathcal{B}_u$.
\item[(b)] Conversely, each component of $X_w$, where $w\in W(T)=S_n$ is an involution of the shape $w=[12...][...]...[...n]$ (the block $[...]$ means the reversion along the word), is a dense open subset of a component of some Springer fibre.
\end{itemize}
\end{thm}

\begin{proof}
{\bf Proof of (a).}

\vspace{2mm} We first prove (a) for the Weyr form $W(u)$ of $u$, in a constructible manner. 

\vspace{2mm} Let us construct a special $\lambda(u)$-tableau $\mathcal{T}$ which will be critical for us: This is done by filling $\{1,...,n\}$ into $\lambda(u)$ in the way that, first fill the 1st column of $\lambda(u)$ from up to down, and then the 2nd column of $\lambda(u)$ from up to down, and so on. So we have 
\begin{equation}\label{formula:tableau T}
{\textrm{the $i$-th column of}\ \mathcal{T}}=
\begin{cases}
\left(\sum_{j<i} c_j\right)+1 \\
\left(\sum_{j<i} c_j\right)+2 \\
...\\
...\\
\left(\sum_{j<i} c_j\right)+c_i
\end{cases}.
\end{equation}
(Recall that we have made the convention $c_0:=0$.)

\vspace{2mm} By the Robinson--Schensted correspondence (Proposition~\ref{prop:RS}) this gives an involution 
\begin{equation}\label{formula:beta(u)}
\beta(u):=w(\mathcal{T},\mathcal{T})=\prod_{i=1}^{r_1}R_i,
\end{equation}
where $R_i$ is the reversion along the $i$-th column of $\mathcal{T}$:
\begin{equation*}
R_i=\begin{pmatrix}
\left(\sum_{j<i}c_j\right)+1    & \left(\sum_{j<i}c_j\right)+2    & \cdots &  \left(\sum_{j<i}c_j\right)+c_i \\
\left(\sum_{j<i}c_j\right)+c_i & \left(\sum_{j<i}c_j\right)+c_i-1 & \cdots & \left(\sum_{j<i}c_j\right)+1
\end{pmatrix}.
\end{equation*}

Now consider the corresponding Deligne--Lusztig variety $X_{\beta(u)}=X_{w(\mathcal{T},\mathcal{T})}$; we want to compute it using flags via \hyperlink{condition (*)}{{\bf (*)}}. Let us fix a set $\{ e_1,...,e_n\}$ as a standard basis of $V$ (over $\mathbb{F}_q$), and view $B$ (resp.\ $T$) as the standard upper triangular subgroup (resp.\ diagonal maximal torus) with respect to this basis. Then the Weyl group $W(T)$ is identified as the symmetric group $S_n$ permuting the subscripts $\{ 1,...,n \}$ of the basis and the simple reflections are identified as the transpositions $(i,i+1)$. Then by \hyperlink{condition (*)}{{\bf (*)}}, a complete flag 
$$\mathcal{F}\colon V_0\subseteq ... \subseteq V_n=V\in G/B$$
lies in $X_{w(\mathcal{T},\mathcal{T})}$ if and only if the following two conditions hold: 
\begin{itemize}
\item[(i)] For each $c_i$, one has 
$$FV_{c_0+c_1+...+c_i}=V_{c_0+c_1+...+c_i}$$
(i.e.\ the spaces $V_{c_0+...+c_i}$ are $F$-stable);
\item[(ii)] for any $s,t$ in the interval $(c_0+...+c_{i-1},c_0+...+c_{i})$, one has
$$\dim V_{s}\cap FV_{t}=\# \left( \{1,...,s\}\cap \left\{w(\mathcal{T},\mathcal{T})(1),...,w(\mathcal{T},\mathcal{T})(t)\right\} \right).$$
\end{itemize}
For each $R_i$, let $\overline{R}_i$ be the corresponding reversion ``modulo $\sum_{j<i}c_j$'', namely 
\begin{equation*}
\overline{R}_i=\begin{pmatrix}
1    & 2    & \cdots &  c_i \\
c_i & c_i-1 & \cdots & 1
\end{pmatrix}.
\end{equation*}
Then the condition (ii) can be re-written as:
\begin{itemize}
\item[(ii)'] For any $s,t$ in the interval $(c_0+...+c_{i-1},c_0+...+c_{i})$, one has
$$\dim \overline{V}_{s}\cap F\overline{V}_{t}=\# \left( \left\{1, 2 , ..., s-\sum_{j<i}c_j \right\}\bigcap  \left\{ \overline{R}_{i}(1), \overline{R}_{i}(2), ..., \overline{R}_{i}\left( t-\sum_{j<i}c_j \right) \right\} \right),$$
where $\overline{V}_{s}:=V_s/V_{c_0+...+c_{i-1}}$ and $\overline{V}_t:=V_{t}/V_{c_0+...+c_{i-1}}$.
\end{itemize}

\vspace{2mm} Note that, considering in the quotient space $V_{c_0+...+c_{i}}/V_{c_0+...+c_{i-1}}$, the condition (ii)' can be viewed as the flag condition for the Deligne--Lusztig variety $X_{\overline{R}_i}$ of 
$$G(i):=\mathrm{GL}(V_{c_0+...+c_{i}}/V_{c_0+...+c_{i-1}})$$
at $\overline{R}_i$. Therefore, the conditions (i) and (ii)' tell us that, to construct a flag $\mathcal{F}\colon V_0\subseteq ... \subseteq V_n$ in $X_{w(\mathcal{T},\mathcal{T})}$ is the same as to do the following two steps:
\begin{itemize}
\item[(iii)] Construct an $F$-stable partial flag 
$$\{0\}=V_{c_0}\subseteq V_{c_0+c_1}\subseteq V_{c_0+c_1+c_2} \subseteq ...\subseteq V_{c_0+...+c_{r_1}}=V$$ 
of spaces of dimensions $c_0,c_0+c_1,...,c_0+...+c_{r_1}=n$, which is the same as to choose a point in $(G/P)^F=G^F/P^F$, where $P$ is the standard parabolic subgroup fixing the partial flag
$$\langle e_1, ..., e_{c_1} \rangle \subseteq \langle e_1, ..., e_{c_1+c_2} \rangle \subseteq ... \subseteq \langle e_1, ..., e_{c_1+...+c_{r_1}}=e_n \rangle=V;$$
\item[(iv)] for each quotient space $V_{c_0+...+c_{i}}/V_{c_0+...+c_{i-1}}$, take a point in the Deligne--Lusztig variety $X_{\overline{R}_i}$.
\end{itemize}

\vspace{2mm} Therefore we get a decomposition
\begin{equation}\label{formula:component decomposition}
X_{w(\mathcal{T},\mathcal{T})}\cong \bigsqcup_{(G/P)^F} X_{\overline{R}_1}\times ... \times X_{\overline{R}_{r_1}}
\end{equation}
into closed subvarieties. (To fix a partial part of a complete flag in the form given in the step (iii) is the same as to require a Borel subgroup to be inside a fixed parabolic subgroup, hence constitute a closed condition.) Note that $\overline{R}_i$ corresponds to the longest element of the Weyl group of  $G(i)$, so $X_{\overline{R}_i}$ is an open subvariety of the flag variety of $G(i)$, hence irreducible. Thus the disjoint union \eqref{formula:component decomposition} is actually the component decomposition of $X_{w(\mathcal{T},\mathcal{T})}$, and $G^F$ acts transitively on the components by translating the partial flag corresponding to $P$.

\vspace{2mm} Denote by $C_1$ the component of $X_{w(\mathcal{T},\mathcal{T})}$ corresponding to $1\cdot P^F\in G^F/P^F$. By construction $C_1$ consists of the flags $\mathcal{F}\colon V_0\subseteq ... \subseteq V_n$ satisfying that
\begin{itemize}
\item[(v)] $V_{c_0+...+c_i}=\langle e_1, ..., e_{c_0+...+c_i} \rangle$ for all $i$;
\item[(vi)] $=$(ii).
\end{itemize} 
We shall show that $C_1\subseteq C(\mathcal{T})$, where the latter is an open subset of the component of $\mathcal{B}_{W(u)}$ corresponding to $\mathcal{T}$ (see Proposition~\ref{prop:Spal--Stei}). Consider the nilpotent element $N=W(u)-I$, which is a blocked matrix of shape
\begin{equation*}
\begin{bmatrix}
0 & \left(\begin{smallmatrix} I\\ 0 \end{smallmatrix}\right) &  0 & 0 & ... & 0\\
0 & 0 & \left(\begin{smallmatrix} I\\ 0 \end{smallmatrix}\right) & 0 & ... & 0\\
0 & 0 & 0 & \left(\begin{smallmatrix} I\\ 0 \end{smallmatrix}\right) & ... & 0\\
... & ... & ...  & ... & ... & ... \\
0 & ... & 0  & 0 & 0 & \left(\begin{smallmatrix} I\\ 0 \end{smallmatrix}\right) \\
0 & 0 & ... & 0 & 0 & 0\\
\end{bmatrix},
\end{equation*}
where the $I$'s are the identity matrices of sizes $c_2, c_3, ..., c_{r_1}$ and the $0$'s are the zero matrices of possibly different suitable sizes. Note that 
$$NV_{c_0+...+c_i}\subseteq V_{c_0+...+c_{i-1}}$$
for any flag $\mathcal{F}\colon V_0\subseteq ... \subseteq V_n$ in $C_1$. Thus by Proposition~\ref{prop:Spal--Stei}, to show that $C_1\subseteq C(\mathcal{T})$, it is sufficient to show that every $\mathcal{F}\in C_1$ satisfies:
\begin{itemize}
\item[(vii)] If $s$ is in the interval $(c_0+...+c_{i-1},c_0+...+c_{i})$, then 
\begin{equation*}
\begin{cases}
\mathrm{Ker}N^{i-1}|_{V_{s+1}}\subseteq V_s  \\
\mathrm{Ker}N^{i}|_{V_{s+1}}\nsubseteq V_s  \\
\end{cases};
\end{equation*}
\item[(viii)] if $s=c_0+...+c_{i}$ (with $i<r_1$), then 
\begin{equation*}
\begin{cases}
\mathrm{Ker}N^{i}|_{V_{s+1}}\subseteq V_s   \\
\mathrm{Ker}N^{i+1}|_{V_{s+1}}\nsubseteq V_s  \\
\end{cases}.
\end{equation*}
\end{itemize} 
Actually this is clear, because a direct blocked-matrix computation gives that
$$\mathrm{Ker}N^i=V_{c_0+...+c_i}$$
for $\forall i\in\{ 1,...,r_1 \}$. So $C_1\subseteq C(\mathcal{T})$.

\vspace{2mm} Next we show that $\mathcal{B}_{W(u)}$ does not intersect any other component of $X_{w(\mathcal{T},\mathcal{T})}$, namely,
$$C_1=\mathcal{B}_{W(u),w(\mathcal{T},\mathcal{T})}.$$
By the decomposition~\eqref{formula:component decomposition}, it is sufficient to show that, if a flag $\mathcal{F}\colon V_0\subseteq ... \subseteq V_n=V$ is stabilised by $u$ (or equivalently, $NV_{i+1}\subseteq V_i$, $\forall i$) and satisfies the condition (ii), then it automatically satisfies the condition (v), i.e.\  $V_{c_0+...+c_i}=\langle e_1, ..., e_{c_0+...+c_i} \rangle$ for all $i$. We first prove this for $V_{c_1}$. Note that, for any $z\in\{ 1,...,c_1-1 \}$, the condition (ii) implies that there is an internal direct sum decomposition
\begin{equation}\label{formula:internal direct sum V_c1}
V_{c_1}=V_z\oplus FV_{c_1-z},
\end{equation}
which gives an internal direct sum decomposition of $V_{z+1}$
\begin{equation*}
V_{z+1}=V_z\oplus F\langle v'\rangle,
\end{equation*}
where $v'$ is some non-zero vector in $V_{c_1-z}$. So by $NV_{z+1}\subseteq V_z$ we get
\begin{equation}\label{formula:internal direct sum V_z}
NV_{z+1}=NV_{z}+ NF\langle v'\rangle \subseteq  V_z.
\end{equation}
Since $N$ is $F$-stable, \eqref{formula:internal direct sum V_z} implies that 
$$FNv'=NFv'\in V_z;$$ 
however, by the internal decomposition \eqref{formula:internal direct sum V_c1}, this happens if and only if $Nv'=0$. Thus from \eqref{formula:internal direct sum V_z} we see that $NV_{z+1}=NV_z$ for any $z\in\{ 1,..., c_1-1\}$. Therefore
$$NV_{c_1}=NV_{c_1-1}=...=NV_1=0,$$
which gives that $V_{c_1}=\langle e_1, e_2,  ..., e_{c_1} \rangle$. Now, suppose $V_{c_0+...+c_i}=\langle e_1, ..., e_{c_0+...+c_i} \rangle$ for some $i$, then by applying the above argument to the quotient space $V_{c_0+...+c_{i+1}}/V_{c_0+...+c_i}$ we also get $V_{c_0+...+c_{i+1}}=\langle e_1, ..., e_{c_0+...+c_{i+1}} \rangle$. So by induction the condition (v) holds, and we conclude that $\mathcal{B}_{W(u)}$ intersects $X_{w(\mathcal{T},\mathcal{T})}$ at exactly the component $C_1$, or more precisely,
\begin{equation}\label{formula:single component}
C_1=\mathcal{B}_{W(u),w(\mathcal{T},\mathcal{T})}=C(\mathcal{T})\cap X_{w(\mathcal{T},\mathcal{T})}.
\end{equation}

\vspace{2mm} It remains to discuss the openness of $C_1$ in the component $\overline{C(\mathcal{T})}$. First note that, since our $N$ is $F$-stable, the step (II) in Proposition~\ref{prop:Spal--Stei} implies that 
$$F {C(\mathcal{T})}={C(\mathcal{T})}.$$ 
So there is a graph embedding
$$(\mathrm{Id}, F)\colon {C(\mathcal{T})} \longrightarrow {C(\mathcal{T})}\times {C(\mathcal{T})}\subseteq G/B\times G/B.$$
It is well-known that (see e.g.\ \cite[7.7]{Carter1993FiGrLieTy}) the subsets $D_v\subseteq G/B\times G/B$ of pairs at various relative positions $v$ give a finite stratification of $G/B\times G/B$ into locally closed subvarieties, so their intersections with ${C(\mathcal{T})}\times {C(\mathcal{T})}$ also give such a stratification of ${C(\mathcal{T})}\times {C(\mathcal{T})}$; let us denote the unique dense open strata of ${C(\mathcal{T})}\times {C(\mathcal{T})}$ by $X$. Then Proposition~\ref{prop:stei} implies that every pair in $X$ is in the relative position $w(\mathcal{T},\mathcal{T})$. In particular we see that
$$(\mathrm{Id}, F)^{-1}(X)={C(\mathcal{T})}\cap X_{w(\mathcal{T},\mathcal{T})},$$
and that this is an open subvariety of $C(\mathcal{T})$ (and hence of $\overline{C(\mathcal{T})}$). So, as $X_{w(\mathcal{T},\mathcal{T})}$ is pure dimensional (see Proposition~\ref{prop:DL property}), by \eqref{formula:single component} we conclude that $C_1$ is a dense open subvariety of $\overline{C(\mathcal{T})}$ as desired.

\vspace{2mm} The above proves (a) for $W(u)$. For a general unipotent $u\in G^F=\mathrm{GL}_n(\mathbb{F}_q)$, by Proposition~\ref{prop:weyr} we know that
$$u^g=W(u),$$
for some $g\in G$. Meanwhile, as both $u$ and $W(u)$ are in $G^F$, by the fact that two matrices similar over a field extension are similar over the original field, we can take $g\in G^F$. (Another way to see this is to apply the Lang--Steinberg theorem, by noting that centralisers in $\mathrm{GL}_{n/\overline{\mathbb{F}}_q}$ are always connected.) Thus $gC_1$ is still a component of $X_{w(\mathcal{T},\mathcal{T})}$ and is also an open dense subset of the component 
$$g\overline{C(\mathcal{T})}\subseteq \mathcal{B}_{{^gW(u)}}=\mathcal{B}_{u}.$$
The uniqueness of the component is also clear by taking a conjugation. This completes the proof of (a).

\vspace{2mm} {\bf Proof of (b).}

\vspace{2mm} For such an involution $w\in W(T)$, by Proposition~\ref{prop:RS} we see that  $w(\mathcal{T},\mathcal{T})=w$ for some tableau $\mathcal{T}$ of the form specified in \eqref{formula:tableau T}. Let $u$ be any unipotent element making $\mathcal{T}$ a $\lambda(u)$-tableau. Then, as we did in the proof of (a), there is a component $C_1$ of $X_w$ lying as an open dense subvariety in the component $\overline{C(\mathcal{T})}$ of $\mathcal{B}_{W(u)}$. 

\vspace{2mm} Now, from the discussion of $X_{w(\mathcal{T},\mathcal{T})}$ given in (a) we know that the translation action of $G^F$ on $X_w$ is transitive on the components; in particular, each component of $X_w$ is of the form $gC_1$ for some $g\in G^F$, which is then an open dense subset of some component of $\mathcal{B}_{{^gW(u)}}$. This completes the proof.
\end{proof}

Recall the dimension formula:

\begin{coro}\label{coro:dim of unip centraliser}
Let $u$ be a unipotent element, then
$$\dim Z_G(u)=\sum_i c_i^2,$$
where $c_i$ is the number of boxes in the $i$-th column of the Young diagram of $u$.
\end{coro}
This can be proved by a matrix manipulation together with a combinatorial consideration (see  \cite[IV.1]{Springer--Steinberg_1970_Conj} and \cite[1.2 and 1.3]{Humphreys_1995_Conj_ss_algGp}); here we derive a geometric proof from the argument of Theorem~\ref{thm:main1}.
\begin{proof}
By Proposition~\ref{prop:weyr} we can assume that $u=W(u)$. Then from the proof of Theorem~\ref{thm:main1} we see that
$$\dim \mathcal{B}_u=\dim X_{w(\mathcal{T},\mathcal{T})}.$$
By Proposition~\ref{prop:DL property} and Proposition~\ref{prop:spr property} this equality can be written as
$$\frac{n(n-1)}{2}-\frac{1}{2}\dim C(u)=l(w(\mathcal{T},\mathcal{T})),$$
where $C(u)$ denotes the conjugacy class of $u$. Note that, by the component decomposition in the argument of Theorem~\ref{thm:main1} (see \eqref{formula:component decomposition}), we have
$$l(w(\mathcal{T},\mathcal{T}))=\sum_il(\overline{R}_i)=\sum_i\frac{c_i(c_i-1)}{2}.$$
So
$$\dim Z_G(u)=\dim G-\dim C(u)=n+\sum_ic_i(c_i-1)=\sum_ic_i^2,$$
as desired.
\end{proof}

\section{Examples and remarks}\label{sec:further remarks}

In this section we give some examples and remarks related to our main theorem, and give a short discussion on the representations associated with $\mathcal{B}_{u,w}=\mathcal{B}_u\cap X_w$ for rectangular $u$.

\vspace{2mm} 1. 

\vspace{2mm} One may wonder that, in Theorem~\ref{thm:main1}, does (a) applies to an arbitrary component of $\mathcal{B}_u$, or does (b) applies to an arbitrary involution of $W(T)$? As illustrated below, without weakening the assertions usually the answer is no:

\begin{expl}\label{expl: illustrating expl}
Let $\mathbb{G}=\mathrm{GL}_4$ and let
\begin{equation*}
u=W(u)=
\begin{bmatrix}
1 & 0 & 1 & 0   \\
0 & 1 & 0 & 1   \\
0 & 0 & 1 & 0  \\
0 & 0 & 0 & 1
\end{bmatrix}\in \mathrm{GL}_4({\mathbb{F}}_q).
\end{equation*}
By Proposition~\ref{prop:spr property} and Proposition~\ref{prop:Spal--Stei}, the variety $\mathcal{B}_{u}$ has two irreducible components, both of dimension $2$, indexed by the Young tableaux
\begin{equation*}
P:=\begin{ytableau}
1 & 3\\
2 & 4
\end{ytableau}
\quad\textrm{and}\quad
Q:=\begin{ytableau}
1 & 2\\
3 & 4
\end{ytableau},
\end{equation*}
respectively. The Robinson--Schensted correspondence gives two involutions,
\begin{equation*}
w\left(P,P\right)
=2143=(1,2)(3,4)
\end{equation*}
and
\begin{equation*}
w\left(Q,Q\right)
=3412=(2,3)(1,2)(3,4)(2,3).
\end{equation*}
(The second equalities are for writing the elements as reduced products of simple reflections.) First consider the component $\overline{C(P)}$. Explicitly, by computing the flag condition \hyperlink{condition (*)}{{\bf (*)}} for the Deligne--Lusztig variety $X_{2143}$, and by computing the step (II) of Proposition~\ref{prop:Spal--Stei} for $C(P)$, one can see that
$$X_{2143}\cong \mathrm{Gr}(2,4)^F\times\left(\mathbb{P}^1\backslash\mathbb{P}^1(\mathbb{F}_q)\right)^2$$
and
$$C(P)\cong (\mathbb{P}^1)^2,$$
which suggests that $C(P)$ contains a component of $X_{2143}$ as a dense open subset, and from the argument of Theorem~\ref{thm:main1} we know that this is exactly the case, and this component is given as $\mathcal{B}_{u,2143}=\mathcal{B}_u\cap X_{2143}$.

Now consider the other component $\overline{C(Q)}$. As mentioned in the argument of Theorem~\ref{thm:main1}, Proposition~\ref{prop:stei} implies that the variety of pairs of flags at relative position $w=3412$ cut out an open dense subset $X$ of ${C(Q)}\times {C(Q)}$, and so the preimage $ {C(Q)}\cap X_{3412}$ of $X$ along the Frobenius graph embedding
$$(\mathrm{Id}, F)\colon  {C(Q)}\longrightarrow  {C(Q)}\times  {C(Q)}$$
is an open subset of $\overline{C(Q)}$; indeed, this open subset is non-empty: By fixing a standard basis $\{ e_1, e_2, e_3, e_4 \}$ of $V=\overline{\mathbb{F}}_q^4$ (over $\mathbb{F}_q$), one easily checks that it contains the flag
$$\{0\}\subseteq \{ e_1+xe_2 \}\subseteq \{  e_1+xe_2,  e_3+xe_4 \} \subseteq \{  e_1, e_2,  e_3+xe_4 \}\subseteq V$$
for any $x\in\overline{\mathbb{F}}_q\backslash\mathbb{F}_q$. This means that some component of $X_{3412}$ cuts out an open dense subset of $\overline{C(Q)}$. However, since
$$\dim X_{3412}=l((2,3)(1,2)(3,4)(2,3))=4>2=\dim\mathcal{B}_u,$$
and since the Deligne--Lusztig varieties partition the flag variety, the component $\overline{C(Q)}$ cannot contain any component of a Deligne--Lusztig variety as an open dense subset. On the other hand, by Proposition~\ref{prop:spr property} and Corollary~\ref{coro:dim of unip centraliser} we see that the possible dimensions of Springer fibres for $\mathrm{GL}_4$ are $0,1,2,3,6$, so none of the components of $X_{3412}$ can be an open subset of a Springer fibre.
\end{expl}

\vspace{2mm} 2.

\vspace{2mm} It is a natural desire that, the map $\beta(-)$ is uniquely characterised by the property given in Theorem~\ref{thm:main1}. To make this true one shall add further conditions:

\begin{expl}
Let $G=\mathrm{GL}_3$ and let 
$$u=W(u)=
\begin{bmatrix}
1 & 0 & 1   \\
0 & 1 & 0   \\
0 & 0 & 1 
\end{bmatrix}\in\mathrm{GL}_3(\mathbb{F}_q).$$
Then the components of $\mathcal{B}_u$  are labelled by the two tableaux of hook shape
\begin{equation*}
P:=\begin{ytableau}
1 & 3\\
2 
\end{ytableau}
\quad\textrm{and}\quad
Q:=\begin{ytableau}
1 & 2\\
3
\end{ytableau},
\end{equation*}
corresponding to the simple reflections $(1,2)$ and $(2,3)$, via the Robinson--Schensted correspondence, respectively. The three varieties $\mathcal{B}_u$, $X_{(1,2)}$, and $X_{(2,3)}$, are all of dimension $1$. From the argument of Theorem~\ref{thm:main1} we know that $C(P)$ contains a component of $X_{(1,2)}$ as a dense open subset and $\mathcal{B}_{u}$ does not intersect other components of $X_{(1,2)}$. Let us consider $X_{(2,3)}$. Given a flag $V_0\subseteq V_1\subseteq V_2\subseteq V_3\subseteq V$, by direct computations with the flag condition \hyperlink{condition (*)}{{\bf (*)}} and the condition that $N=u-I$ takes $V_i$ into $V_{i-1}$, one sees that $\mathcal{B}_u$ intersects $X_{(2,3)}$ only at the component (of $X_{(2,3)}$) consisting of the flags
$$\{0\}\subseteq \{ e_1 \}\subseteq \{  e_1, e_2+\lambda e_3 \} \subseteq \{  e_1, e_2,  e_3 \}=V,$$
where $\lambda$ runs over $\overline{\mathbb{F}}_q\backslash\mathbb{F}_q$. Meanwhile, it follows from the step (II) of Proposition~\ref{prop:Spal--Stei} that this component is contained in $C(Q)$. So, at this $u$, there are two choices of the value of $\beta$ fulfilling the requirement in Theorem~\ref{thm:main1}. 
\end{expl}

Thus we hope to state here a question: 

\begin{quest}
Is there a geometric property (in addition to the one asserted in Theorem~\ref{thm:main1}, but without referencing to the explicit construction \eqref{formula:beta(u)} given in its argument) making the map $\beta$ unique?
\end{quest}

\vspace{2mm} 3.

\vspace{2mm} Let $w_0$ be the longest element of $W(T)$, then $X_{w_0}$ contributes a generic (i.e.\ open dense) part of the flag variety. Quite opposite to the component containment relation in Theorem~\ref{thm:main1}, Springer fibres missed this ``largest'' Deligne--Lusztig variety at all:

\begin{expl}\label{expl:longestWeyl empty}
In this example we show that $\mathcal{B}_{u,w_0}=\mathcal{B}_u\cap X_{w_0}$ is always empty unless $u=1$. Suppose that $\mathcal{B}_{u,w_0}$ is non-empty; let $\mathcal{F}\in \mathcal{B}_{u,w_0}$ be a point in the component $\overline{C(P)}$ for some $\lambda(u)$-tableau $P$. Consider the Frobenius graph embedding 
$$(\mathrm{Id},F)\colon \overline{C(P)}\longrightarrow \overline{C(P)}\times \overline{C(P)}.$$ 
(Here $F$ preserves $\overline{C(P)}$ because $F$ preserves $C(P)$ and $F$ is a homeomorphism.) By Proposition~\ref{prop:stei} and by considering Bruhat order we get
$$``\textrm{the relative position of}\ (\mathcal{F},F\mathcal{F})"=w_0\leq w(P,P),$$
which in turn implies that $w(P,P)=w_0$. So, via the Robinson--Schensted correspondence we see that $\lambda(u)$ must be a single column diagram, that is, $u=1$, in which case $\mathcal{B}_u$ is the whole flag variety.
\end{expl}

\vspace{2mm} 4.

\vspace{2mm} Recall that (see Notation~\ref{notation:unip notations}) we have let $J(u)$ be the Jordan normal form of $u$, $d$ the number of Jordan blocks in $J(u)$, $r_i$ the sizes of each Jordan block, and $\lambda(u)$ the associated Young diagram.

\begin{defi}
A unipotent element $u\in G$ is called rectangular, if the Young diagram $\lambda(u)$ is rectangular.
\end{defi}

Note that if $u$ is rectangular, then $d$ ($=c_i, \forall i$) is a divisor of $n$ and $n/d=r_i$ for all $i$; when this is the case we denote $r_i$ by $r$. 

\vspace{2mm} As mentioned in the introduction, our original focus on the relations between Deligne--Lusztig varieties and Springer fibres comes from the smooth representation theory of the profinite group $\mathrm{GL}_{d}(\mathbb{F}_q[[\pi]])$. Indeed, if $u\in \mathcal{U}^F$ is rectangular, then the finite quotient group $\mathrm{GL}_{d}(\mathbb{F}_q[[\pi]]/\pi^r)$ acts on $\mathcal{B}_{u,w}$, because $\mathrm{GL}_{d}(\overline{\mathbb{F}}_q[[\pi]]/\pi^r)$ is naturally isomorphic to the $G$-centraliser of 
\begin{equation*}
W(u)=\begin{bmatrix}
I_d & I_d &  0 & 0 & ... & 0\\
0 & I_d & I_d & 0 & ... & 0\\
0 & 0 & I_d & 0 & ... & 0\\
... & ... & ...  & ... & ... & ... \\
0 & ... & 0  & 0 & I_d & I_d \\
0 & 0 & ... & 0 & 0 & I_d\\
\end{bmatrix},
\end{equation*}
where $I_d$ denotes the $d\times d$ identity matrix, by the ring injection from $M_d(\overline{\mathbb{F}}_q[[\pi]]/\pi^r)$ to ${M}_{n}(\overline{\mathbb{F}}_q)$:
\begin{equation}\label{formula:injection}
A_0+A_1\pi+...+A_{r-1}\pi^{r-1}\longmapsto
\begin{bmatrix}
A_0     & A_1     & ... & A_{r-2} & A_{r-1}  \\
0       & A_0     & A_1 & ...     & A_{r-2}   \\
...     & ...     & ... & ...     & ... \\
0       & ...     & 0   & A_0     & A_1   \\
0       & 0       & ... & 0       & A_0
\end{bmatrix},
\end{equation}
where $A_i\in M_d(\overline{\mathbb{F}}_q)$. (See \cite[5.3]{Chen_2019_flag_orbit} and \cite[4.6]{Chen_2020_twistetcent}; note that the notation used there is different up to a blocked transpose.) Thus we get a virtual representation
$$R_{u,w}:=\sum_i(-1)^iH_c^i(\mathcal{B}_{u,w},\overline{\mathbb{Q}}_{\ell})$$
of $\mathrm{GL}_{d}(\mathbb{F}_q[[\pi]]/\pi^r)$. Here $H_c^i(-,\overline{\mathbb{Q}}_{\ell})$ denotes the $i$-th compactly supported $\ell$-adic cohomology with $\ell$ a prime not equal to $\mathrm{char}(\mathbb{F}_q)$. Note that if $n=d$, then this construction gives the unipotent representations of $\mathrm{GL}_n(\mathbb{F}_q)$ in the sense of \cite[7.8]{DL1976}.

\begin{defi}
A representation of $\mathrm{GL}_{d}(\mathbb{F}_q[[\pi]]/\pi^r)$, where $r\geq 2$, is called primitive, if it does not factor through $\mathrm{GL}_{d}(\mathbb{F}_q[[\pi]]/\pi^{r-1})$.
\end{defi}

For example, when $u\in\mathcal{U}^F$ is rectangular with $d,r\geq 2$, from \cite[4.8]{Chen_2020_twistetcent} we known that $R_{W(u),w}$ is a primitive representation of $\mathrm{GL}_{d}(\mathbb{F}_q[[\pi]]/\pi^r)$ if $w=(1,...,z)$ is a cycle with $z<d$.

\begin{quest}
For a given rectangular $u\in\mathcal{U}^F$ (resp.\ Weyl element $w\in W(T)$), is there a reasonable characterisation of the Weyl element $w$ (resp.\ rectangular $u\in\mathcal{U}^F$) making $R_{u,w}$ primitive?
\end{quest}

We note that there is a uniform description of the representations $R_{u,w}$ for various rectangular $u\in\mathcal{U}^F$, using a complex on $G$: First consider the character-sheaf type diagram (see \cite{Lusztig_CharSh_I})
\begin{equation*}
\begin{tikzcd}
X_w & Z_w \arrow{l}[swap]{b} \arrow{r}{a} & G,
\end{tikzcd}
\end{equation*}
where 
$$Z_w:=\{(g, xB)\in G\times L^{-1}(BwB)/B \mid g^x\in B\},$$
and $a, b$ are the natural projections; this diagram is $G^F$-equivariant, where $G^F$ acts on $Z_w$ by $h\cdot(g,x)=(hgh^{-1},hx)$, on $G$ by left conjugation, and on $X_w$ by left multiplication. Then the character-sheaf type complex 
$$K:=Ra_!b^{*}\overline{\mathbb{Q}}_{\ell}\in D^b_c(G,\overline{\mathbb{Q}}_{\ell})$$ 
encodes all $R_{u,w}$: 

\begin{prop}
If $u\in \mathcal{U}^F$ is rectangular, then $R_{u,w}=\sum_i(-1)^i \mathcal{H}^i(K)_u$.
\end{prop}
\begin{proof}
Note that $\mathcal{B}_{u,w}\cong a^{-1}(u)$, so the assertion formally  follows from the proper base change along $\{u\}\hookrightarrow G$.
\end{proof}

We end with a (non-)smoothness property of the above diagram.

\begin{prop}
The variety $Z_w$ and the morphism $b$ are smooth. However, the morphism $a$ is smooth only for $n=1$.
\end{prop}
\begin{proof}
We follow the idea of \cite[2.5.2]{Lusztig_CharSh_I} to use a faithful flat descent. Let $\widetilde{Z_w}$ be the base change of $Z_w\subseteq G\times L^{-1}(BwB)/B$ along the morphism
\begin{equation*}
G\times L^{-1}(BwB)\longrightarrow G\times L^{-1}(BwB)/B.
\end{equation*}
This is faithfully flat since $B$ is solvable (hence the quotient $L^{-1}(BwB)\rightarrow L^{-1}(BwB)/B$ admits local sections). Then by \cite[17.7.7]{grothendieck_elements_1967} it suffices to prove that the variety $\widetilde{Z_w}=\{(g,x)\in G\times L^{-1}(BwB)\mid g^x\in B\}\subseteq G\times G$ and the morphism $\widetilde{b}\colon\widetilde{Z_w}\rightarrow X_w$ (extending $b$) are smooth. Applying the variable change $y=g^x$ we get
\begin{equation*}
\widetilde{Z_w}\cong \{(y,x)\in G\times L^{-1}(BwB)\mid y\in B\}=B\times L^{-1}(BwB),
\end{equation*}
and $\widetilde{b}$ reads as $(y,x)\mapsto xB$. As the Lang morphism $L$ is \'etale, the assertion on $Z_w$ and $b$ follows.

\vspace{2mm} For the assertion on $a$, again by the faithful flat descent it suffices to show that the morphism $\widetilde{Z_w}\rightarrow G$ given by
$$\widetilde{a}\colon (y,x)\in B\times L^{-1}(BwB)\longmapsto xyx^{-1}$$
is not smooth. Actually this morphism is not flat: Note that the fibre of $\widetilde{a}$ at any central element $c\in Z(G)$ is $\{c\}\times L^{-1}(BwB)\subseteq B\times L^{-1}(BwB)$. So, if $\widetilde{a}$ is flat then we have (see e.g.\ \cite[4.3.12]{liu_algebraic_2006})
$$\dim L^{-1}(BwB)=\dim B+ \dim L^{-1}(BwB)-\dim G,$$
which is impossible unless $n=1$.
\end{proof}

In the above argument, note that for any $b\in B$, the elements $(y,x)$ and $(b^{-1}yb,xb)$ have the same image under $\widetilde{a}$. So, the fibre of $\widetilde{a}$ at a closed point is either empty or has dimension at least $\dim B$, which implies that (via \cite[4.3.12]{liu_algebraic_2006}) 
$$\dim B\leq \widetilde{a}^{-1}(xyx^{-1})= \dim B+\dim L^{-1}(BwB)-\dim G$$ 
if $\widetilde{a}$ is flat at $(y,x)$. Thus actually $\widetilde{a}$ cannot be flat at any point unless $w$ is the longest element $w_0$.

\bibliographystyle{alpha}
\bibliography{zchenrefs}

\end{document}